\documentclass[11pt]{amsart}
\usepackage{geometry}

\usepackage{fullpage}
\usepackage{amsfonts,amscd}
\usepackage{amssymb}
\usepackage{url}
\usepackage{epsfig}

\usepackage{pstricks}
\usepackage{pst-plot}
\usepackage{pst-node}
\usepackage[fleqn,tbtags]{mathtools}

\def \r{\mathbb R}
\def \s{\mbox{${\mathbb S}$}}

\def \h{\mathbb  H}

\DeclareMathOperator{\arctanh}{arctanh}
\DeclareMathOperator{\arcsinh}{arcsinh}
\DeclareMathOperator{\sech}{sech}

\DeclareMathOperator{\arccot}{arccot}

\theoremstyle{plain}
\newtheorem{theorem}                 {Theorem}      [section]
\newtheorem{proposition}  [theorem]  {Proposition}
\newtheorem{corollary}    [theorem]  {Corollary}
\newtheorem{lemma}        [theorem]  {Lemma}

\theoremstyle{definition}
\newtheorem{example}      [theorem]  {Example}
\newtheorem{remark}       [theorem]  {Remark}

\begin{document}


\title{Loxodromes on invariant surface in three-manifolds}

\author{Renzo Caddeo}
\address{Universit\`a degli Studi di Cagliari\\
Dipartimento di Matematica e Informatica\\
Via Ospedale 72\\
09124 Cagliari, Italia}

\email{caddeo@unica.it}

\author{Irene I. Onnis}

\address{Departamento de Matem\'{a}tica, C.P. 668\\ ICMC,
USP, 13560-970, S\~{a}o Carlos, SP\\ Brasil}

\email{onnis@icmc.usp.br}

\author{Paola Piu}
\address{Universit\`a degli Studi di Cagliari\\
Dipartimento di Matematica e Informatica\\
Via Ospedale 72\\
09124 Cagliari, Italia}

\email{piu@unica.it}

\thanks{The second author was supported by grant 2016/24707-4, S\~ao Paulo Research Foundation (Fapesp) and by CNPq productivity  grant 312700/2017-2. The last author was supported by PRIN 2015 ``Variet\`a reali e complesse: geometria, topologia e analisi armonica''  Italy; and GNSAGA-INdAM, Italy.}

\begin{abstract}
In this paper, we prove important results  concerning the loxodromes on an invariant surface in a three-dimensional Riemannian manifold, some of which generalize classical results about loxodromes on rotational surfaces in $\r^3$. 
In particular, we show how to parametrize a loxodrome on an invariant surface of  $\h^2\times \r$ and $\h_3$ and we exhibit the loxodromes of some remarkable minimal invariant surfaces of these spaces.  In addition, we give an explicit description of the loxodromes on an invariant surface with
constant Gauss curvature.
\end{abstract}

\maketitle

\section{Introduction and preliminaries}

In this paper, we study the loxodromes on an important family of surfaces in a three-manifold, that of the  surfaces which are invariant under the action of a one-parameter group of isometries of the ambient space. Invariant surfaces have been classified by
Gaussian or mean curvature in many remarkable three-dimensional spaces (see, for example, \cite{RCPPAR1,RCPPAR2,FIMEPE,LO,MO1,MO,Onnis,TO}).  Also, in \cite{MO3, MO2,PiuProfir} have been studied two well known types of curves on invariant surfaces: the geodesics and the proper biharmonic curves.
\\

A rotational surface of the Euclidean three-space is $SO(2)$-invariant and a loxodrome is a curve on it which meets the meridians  at a constant angle. In \cite{N}, C.A. Noble obtained the differential equations of a loxodrome on these surfaces and, in particular, he investigated these curves on spheres and spheroids.
Recently, the results of \cite{N} have been generalized by Babaarslan and Yayli to the case of helicoidal surfaces in $\r^3$ (see \cite{Ba}).\\

As the meridians and the parallels of a rotational surface of $\r^3$ are orthogonal, the loxodromes may also be defined as curves that make a constant angle with the Killing vector field $X$ that is the infinitesimal generator of the one-parameter subgroup of isometries given by $SO(2)$. 
Therefore, if we denote by $G_X$  the one-parameter subgroup of isometries of the ambient space generated by the Killing vector field $X$, the loxodromes on a $G_X$-invariant surface  can naturally defined as the curves which make a constant angle with $X$.\\

In order to investigate loxodromes on an invariant surface in a three-dimensional manifold, we need to recall some basic fact on the geometry of invariant surfaces.
\subsection{Equivariant geometry of invariant surfaces}
Let
$({N}^3,g)$ be a three-dimensional Riemannian manifold and let $X$
be a Killing vector field on ${N}$. Then $X$ generates a
one-parameter subgroup $G_X$ of the group of isometries of
$({N}^3,g)$. Let now $f:{M}\to ({N}^3,g)$ be an immersion from a
surface ${M}$ into ${N}^3$ and assume that $f({M})\subset {N}_r$
(the regular part of $N$,  that is, the subset consisting of
points belonging to principal orbits). We say that $f$ is a $G_X$-{\it
equivariant immersion}, and $f({M})$ a $G_X$-{\it invariant surface}
of ${N}$, if there exists an action of $G_X$ on ${M}$ such that
for any $x\in {M}$ and $a\in G_X$ we have $f(a\,x)=a f(x)$.
A $G_X$-equivariant immersion $f:{M}\to ({N}^3,g)$ induces on ${M}$ a Rieman\-nian metric,
the pull-back metric, denoted by $g_f$ and called the $G_X$-{\it invariant induced metric}.\\

Let $f:{M}\to ({N}^3,g)$ be a $G_X$-equivariant immersion and let $g_f$ be the
 $G_X$-invariant induced metric on $M$.
Assume 
that ${N}/G_X$ is connected. Then, $f$ induces an
immersion $\tilde{f}:{M}/G_X\rightarrow {N}_r/G_X$ between the orbit
spaces and, moreover, ${N}_r/G_X$ can be equipped with a
Riemannian metric, the {\it quotient metric}, so that the quotient
map $\pi:{N}_r\to {N}_r/G_X$ becomes  a Riemannian submersion.
Thus we have the following diagram
$$
\begin{CD}
({M},g_f) @>f>> ({N}^3_{r},g)\\
@V VV @V\pi VV\\
{M}/G_X @>\tilde{f}>> ({N}^3_{r}/G_X,\tilde{g}).
\end{CD}
$$

For later use, we describe the quotient metric of the regular part of the orbit space
$N_r/G_X$. It is well known (see, for example, \cite{Olver})  that
$N_r/G_X$ can be locally parametrized by the invariant functions of
the Killing vector field $X$.  If $\{\xi_1,\xi_2\}$ is a complete set of
invariant functions on a $G_X$-invariant subset  of $N_r$, then the
quotient metric is given by  $\tilde{g}=~\sum_{i,j=1}^{2} h^{ij}
d\xi_i\otimes d\xi_j$ where $(h^{ij})$ is the inverse of the matrix
$(h_{ij})$ with entries $h_{ij}=g(\nabla \xi_i,\nabla \xi_j)$.\\

Now, we can give a local description of the $G_X$-invariant surfaces of
${N}^3$. Let $\tilde{\gamma}:(a,b)\subset\r\to({N}^3/G_X,\tilde{g})$
be a curve parametrized by arc length and let
$\gamma:(a,b)\subset\r\to {N}^3$ be a lift of $\tilde{\gamma}$, such
that $d\pi(\gamma')=\tilde{\gamma}'$. If we denote by
$\phi_v,\;v\in(-\epsilon,\epsilon)$, the local flow of the Killing
vector field $X$, then the map
\begin{equation}\label{eq-psi}
\psi:(a,b)\times(-\epsilon,\epsilon)\to {N}^3\,,\quad \psi(u,v)=\phi_v(\gamma(u)),
\end{equation}
defines a parametrization of a
$G_X$-invariant surface. Conversely, if $f({M})$ is a $G_X$-invariant immersed surface in ${N}^3$,
then $\tilde{f}$ defines a curve in $({N}^3/G_X,\tilde{g})$ that can be locally
parametrized by arc length. The curve $\tilde{\gamma}$ is generally called the {\it profile 
curve} of the invariant surface.\\

Observe that, as the $v$-coordinate curves are the orbits of the
action of the one-parameter group of isometries $G_X$, the
coefficients of the pull-back metric $g_f=E\, du^2 +2 F\, du dv + G\,
dv^2$ are function only of $u$ and are given by:
$$
\begin{cases}
E=g(\psi_u,\psi_u)=g(d\phi_v(\gamma'),d\phi_v(\gamma'))\\
F=g(\psi_u,\psi_v)=g(d\phi_v(\gamma'),X)\\
G=g(\psi_v,\psi_v)=g(X,X).\\
\end{cases}
$$
Putting  $(\omega(u))^2:=\|X(\gamma(u))\|^2_g=G$, one has (see~\cite{MO})
\begin{equation}
\label{eq-E} E\,G-F^2=G=(\omega(u))^2.
\end{equation}
\begin{remark}
When $\gamma$ is a
horizontal lift of $\tilde{\gamma}$, we have that $F=0$
and $E=1$. Thus, the equation~\eqref{eq-E} is immediate.
\end{remark}
By using \eqref{eq-E} and Bianchi's formula for the Gauss curvature,
we find that
\begin{equation}\label{eq-K}
K(u)=-\frac{\omega_{uu}(u)}{ \omega(u)}.
\end{equation}
Consequently, we have
\begin{theorem}[\cite{MO}]\label{teo-main}
Let $f:M\to(N^3,g)$ be a $G_X$-equivariant immersion,
$\tilde{\gamma}:(a,b)\subset\r\to(N_r^3/G_X,\tilde{g})$
 a parametrization by arc length of the profile curve of $M$ and $\gamma$ a lift of $\tilde{\gamma}$.
Then the induced metric $g_f$ is of constant Gauss curvature $K$ if
and only if the function $\omega(u)$
 satisfies the  differential equation
\begin{equation}\label{eq-main}
\omega_{uu}(u) + K \omega(u) =0.
\end{equation}

\end{theorem}

\section{A parametric equation for  loxodromes}

In this section, we obtain the equation of  the loxodromes on an invariant surface which are not orbits in terms of the parameters $u$ and $v$.

Let $M$ be a $G_X$-invariant surface of $(N^3,g)$, locally parametrized by
\eqref{eq-psi}; then the induced metric is given by
$$
g_f=E(u)\,du^2+2\,F(u)\,du\,dv+(\omega(u))^2\,dv^2.
$$
Now, let $\alpha(s)=\psi(u(s),v(s))$ be a loxodrome on $M$ parametrized by arc length, so that
 \begin{equation}\label{arc}
   1=g(\alpha',\alpha')= E(u(s))\, u'^2+2\, F(u(s))\,u'\,v'+(\omega (u(s)))^2\,v'^2.
\end{equation}
Denoting by $\vartheta_0\in [0,\pi)$ the constant angle under which the curve $\alpha$ meets the orbits of the Killing vector field $X$, we have
\begin{equation}\label{lox}
\begin{aligned}
\omega(u(s)) \cos\vartheta_0&=g(\alpha',X)= g(\alpha',\psi_v)\\
&=F(u(s))\,{u'}(s)+(\omega(u(s)))^2\,{v'}(s).
\end{aligned}.
\end{equation}
In the following, we suppose that $\vartheta_0\neq 0$. Then, the loxodrome is not an orbit and
$$
\alpha (s) \neq \psi\Big(u_0,\frac{s}{\omega(u_0)}\Big), \qquad u_0\in (a,b).
$$
\begin{remark}
If the loxodrome $\alpha(s)$
 is orthogonal to all the orbits that it meets (i.e. $\vartheta_0=\pi/2$),  then it is a geodesic. In fact, 
we observe that, as $\alpha$ cannot be an orbit, it is a geodesic if satisfies the following system (see \cite{MO2})
\begin{equation}\label{eqgeod-bis}
\left\{\begin{aligned}
\|\alpha'\|&=1,\\
(F\,{u'}+\omega^2\,{v'})'&=0.
\end{aligned}\right.
\end{equation} 
We only have to show that the second equation of \eqref{eqgeod-bis} is satisfied. By using \eqref{lox}, from $g(\alpha',X)=0$ we get $F u'+\omega^2 v'=0$.
\end{remark}

\begin{lemma}\label{tgeod}
Let ${M}\subset ({N}^3,g)$ be a $G_X$-invariant surface locally parametrized by
$\psi(u,v)$ given by \eqref{eq-psi} and let $\alpha (s)=\psi(u(s),v(s))$ be a loxodrome, parametrized by arc length,
which is not an orbit. Then,
 \begin{equation}\label{eqsist}
    \left\{\begin{aligned}
    &F(u(s))\,{u'}+\omega(u(s))^2\,{v'}=\omega(u(s))\,\cos\vartheta_0,\\
    &{u'}^2=\sin^2\vartheta_0.
    \end{aligned}\right.
 \end{equation}
Conversely, if system~\eqref{eqsist} is satisfied and ${u'}\neq 0$, then $\alpha$ is a loxodrome parametrized by arc length.
 \end{lemma}
 \begin{proof}
If $\alpha$ is a loxodrome parametrized by arc length, then condition \eqref{lox}, and therefore the first equation of system~\eqref{eqsist} is satisfied. Also, the vector field
$\alpha'(s)=\psi_u\,u'(s)+\psi_v\,v'(s)$ verifies \eqref{arc}.
Moreover, 
\begin{equation}\label{eqc}\begin{aligned}
2\,F(u(s))\,{u'}\,{v'}+\omega(u(s))^2\,{v'}^2 &= v'(s)\left(F(u(s))\,{u'}+ \omega(u(s)) \cos\vartheta_0\right) \\
&= \left(F(u(s))\,{u'}+ \omega(u(s)) \cos\vartheta_0\right)\;  \frac{\omega(u(s)) \cos\vartheta_0 - F(u(s))\,{u'}}{\omega(u(s))^2} \\
&= \frac{\omega(u(s))^2\,\cos^2\vartheta_0 - F(u(s))^2\,{u'}^2}{\omega(u(s))^2}.
\end{aligned}
\end{equation}

Therefore, by substituting  \eqref{eqc} in \eqref{arc} and using \eqref{eq-E} we obtain the second equation of \eqref{eqsist}:
$$
1=\frac{E(u(s))\,\omega(u(s))^2-F(u(s))^2}{\omega(u(s))^2}\,{u'}^2+\cos^2\vartheta_0={u'}^2+\cos^2\vartheta_0.
$$

Conversely, if system~\eqref{eqsist} is satisfied, from the second equation of  \eqref{eqsist} by taking into account  \eqref{eq-E}  we have 
\begin{equation}\label{eqarc2.bis}\begin{aligned}
1&={u'}^2+\cos^2\vartheta_0=
\bigg(\frac{E(u(s))\,\omega(u(s))^2-F(u(s))^2}{\omega(u(s))^2}\bigg)\,{u'}^2+\cos^2\vartheta_0\\
&=E(u(s))\,{u'}^2+\frac{\omega(u(s))^2\,\cos^2\vartheta_0-F(u(s))^2\,{u'}^2}{\omega(u(s))^2}.
\end{aligned}
\end{equation}
Now we use the first equation of  \eqref{eqsist} to obtain
\begin{equation}\label{eqarc2}\begin{aligned}
1&=
E(u(s))\,{u'}^2+v'\,[\omega(u(s))\,\cos\vartheta_0+F(u(s))\,{u'}]
\\&=E(u(s))\,{u'}^2+2\,F(u(s))\,{u'}\,{v'}+\omega(u(s))^2\,{v'}^2\\&=
g(\alpha',\alpha'),
\end{aligned}
\end{equation}
so that $\alpha$ has unit speed. Moreover, since $\alpha$ is not an orbit, from the first of equation  \eqref{eqsist} we conclude that $\alpha$ is a loxodrome.
 \end{proof}
 
By integrating system~\eqref{eqsist} we have the following
\begin{theorem}
A loxodrome on a $G_X$-invariant surface ${M}\subset ({N}^3,g)$, which is not an orbit, can be locally parametrized by $\beta(u)=\psi(u,v(u))$, where
 \begin{equation}\label{eq-integral}
    v(u)=\int_{u_0}^u \left(\frac{-F}{\omega^2} \pm \frac{\cot\vartheta_0}{\omega}\right)\,dt.
\end{equation}
\end{theorem}
\begin{proof}
We consider the surface $M$ locally parametrized by
$\psi(u,v)$ given by \eqref{eq-psi} and we suppose that $\alpha (s)=\psi(u(s),v(s))$ is a
loxodrome on $M$ that is not an orbit parametrized by arc length. As ${u'}\neq 0$ we can locally invert the function $u=u(s)$ obtaining $s=s(u)$ and  we can consider the parametrization of $\alpha$ given by
$$\beta(u)=\alpha(s(u))=\psi(u,v(u)),\qquad v(u)=v(s(u)).$$

By multiplying  by $(ds/du)^2$  the equation
$$
E(u(s))\, u'(s)^2+2\, F(u(s))\,u'(s)\,v'(s)+\omega (u(s))^2\,v'(s)^2=g(\alpha',\alpha')=1
$$
 we obtain
\begin{equation}\label{eqpri}
E+2\,F\,\frac{dv}{du}+\omega^2\,\Big(\frac{dv}{du}\Big)^2=\Big(\frac{ds}{du}\Big)^2.
\end{equation}
Moreover, from the second equation in \eqref{eqsist}, we get
\begin{equation}\label{eqpri1}
\Big(\frac{ds}{du}_{\big|u(s)}\Big)^2=\frac{1}{{u'}(s)^2}=\frac{1}{(\sin\vartheta_0)^2}.
\end{equation}
By substitution of \eqref{eqpri1} in \eqref{eqpri} gives
\begin{equation}\label{eqpri1-b}\omega^2\,\Big(\frac{dv}{du}\Big)^2+2\,F\,\frac{dv}{du}+E-\frac{1}{(\sin\vartheta_0)^2}=0.
\end{equation}
Now, taking into account \eqref{eq-E}, we observe that
$$
F^2-\omega^2\,\Big[E-\frac{1}{(\sin\vartheta_0)^2}\Big]=(\cot\vartheta_0\,\omega)^2.
$$
Consequently, from \eqref{eqpri1-b} we obtain
$$
   \frac{dv}{du}=\frac{-F\pm\cot\vartheta_0\,\omega}{\omega^2},
$$
from this formula one obtains at once \eqref{eq-integral}, which is the equation of a segment of a loxodrome which is not an orbit.
\end{proof}

As immediate consequence of equation~\eqref{eqpri1} is given by
\begin{corollary}
The length of a loxodrome which is not an orbit on a $G_X$-invariant surface ${M}\subset ({N}^3,g)$  between two orbits $u_1$ and $u_2$ is given by
\begin{equation}\label{arc-lox}
s=\frac{u_2-u_1}{\sin\vartheta_0}.
\end{equation}
\end{corollary}

\subsection{Loxodromes on translational and helicoidal surfaces in $\r^3$}
We consider the Euclidean three-dimensional space $\r^3$ with the canonical metric $g=dx^2+dy^2+dz^2$. Then  the Killing
vector fields generate  translations and rotations. 

In the case of translations along the direction of the unitary Killing vector field $X$, the
quotient space $\r^3/G_{X}$ is the plane $\r^2$ (orthogonal to $X$) equipped with the flat metric. The  invariant surface is a flat right cylinder over the curve 
$$\tilde{\gamma}(u)=(\xi_1(u),\xi_2(u)),\qquad \xi_1'^2+\xi_2'^2=1,$$ 
and it can be parametrized by 
$$\psi(u,v)=\gamma(u)+v\,X,$$ 
where $\gamma$ is a horizontal lift of $\tilde{\gamma}$. Therefore, the coefficients of the induced metric are  $E=1$, $F=0$ and $\omega=1$ and its loxodromes are parametrized by 
\[
\beta(u)=\psi(u,v(u))=\gamma(u) \pm \cot\vartheta_0\,(u-u_0)\,X.
\]
So these curves form an angle $\vartheta_0$ with the Killing vector field $X$, i.e. they are general helices whose axis is $X$.  

In the case of the helicoidal surfaces we can assume, without loss of generality, that the Killing vector field is
$X=-y\frac{\partial}{\partial x}+x\frac{\partial}{\partial y}+a \frac{\partial}{\partial z}$, with $a\in\r$. Introducing cylindrical coordinates $(r,\theta,z)$, we have  
\[
X=\frac{\partial}{\partial \theta}+a\,\frac{\partial}{\partial z}, \qquad
 \qquad
g =dr^2 + r^2 d \theta^2 + dz^2 ,
\]

and $\{\xi_1=r,\xi_2=z-a\,\theta\}$ is a set of independent $G_X$-invariant functions. Therefore, the regular part of the orbit space
 is  $$\r^3_r/G_{X}=\{(\xi_1,\xi_2)\in\r^2\;:\;\xi_1>0\},$$ 
with respect to its orbital metric (see \cite{docarmo},\cite{FIMEPE})
  $$
\tilde{g}=d\xi_1^2+\frac{\xi_1^2}{a^2+\xi_1^2}\,d\xi_2^2,
$$
 the  projection
$$
(r,\theta,z)\xmapsto{\pi}(r,z-a\,\theta)
$$
becomes a Riemannian submersion.
Now we suppose that  the profile curve of a $G_X$-invariant surface $\tilde{\gamma}(u)=(\xi_1(u),\xi_2(u))\in \r^3_r/G_{X}$,  is parametrized by arc length, so that \begin{equation}\label{ppca}
\xi_1'^2(u)+\frac{\xi_1(u)^2}{a^2+\xi_1(u)^2}\,\xi_2'(u)^2=1.
\end{equation}  
Then the norm of $X$ restricted to the profile curve is $\omega(u)=\sqrt{r(u)^2+a^2}$ and a lift of $\tilde{\gamma}$ with respect to $\pi$ is $\gamma(u)=(\xi_1(u),0,\xi_2(u))$. Therefore, the coefficients of the induced metric of the helicoidal surfaces 
$$\psi(u,\theta)=(\xi_1(u)\,\cos \theta, \xi_1(u)\,\sin \theta, \xi_2(u)+a\,\theta)$$ 
are

\begin{equation}\label{star}
E=\xi_1'(u)^2+\xi_2'(u)^2,\qquad F=a\,\xi_2'(u),\qquad G=\omega(u)^2.
\end{equation}

From \eqref{ppca} we get
\begin{equation}\label{eq-E-b}
\xi_2'(u)^2 = \frac{\omega(u)^2}{r(u)^2}(1 - r'(u)^2)
\end{equation} 
and then, from \eqref{star},
\[
\frac{F}{\omega^2}=\frac{a\,\sqrt{1-\xi_2'(u)^2}}{r(u)\,\sqrt{a^2+ r(u)^2}}.
\]
By using \eqref{eq-integral} we conclude that the loxodromes on  helicoidal surfaces can be parametrized by
$\beta(u)=\psi(u,\theta(u)),$
where 
$$\theta(u)=\int_{u_0}^u\Big(-\frac{a\,\sqrt{1-r'(t)^2}}{r(t)\,\sqrt{a^2+r(t)^2}}\pm\frac{\cot\vartheta_0}{\sqrt{a^2+r(t)^2}}\Big)\,dt.$$

\begin{remark}
If $\gamma$ is a
horizontal lift of $\tilde{\gamma}$ (i.e. $F=0$
and $E=1$), equation~\eqref{eq-integral} reduces to
$$\theta(u)=\pm\int_{u_0}^u\frac{\cot\vartheta_0}{\sqrt{a^2+r(t)^2}}\,dt.$$
\end{remark}

We observe that, in the case of the surfaces of revolution  (i.e. for  $a=0$),
$$\beta(u)=(\xi_1(u)\,\cos \theta(u), \xi_1(u)\,\sin \theta(u), \xi_2(u)),
$$
with \begin{equation}
\label{eq-rot}\theta(u)=\pm\cot\vartheta_0 \int_{u_0}^u \frac{dt}{\xi_1(t)}.
\end{equation}
\begin{example}[Sphere]
The unit sphere $\s^3$ is obtained by choosing $a=0$ and $\tilde{\gamma}(u)=(\sin u,\cos u)$, $u\in (0,\pi)$. In this case, equation~\eqref{eq-rot} gives 
$$\theta(u)=\cot\vartheta_0\,\ln\Big(\tan\Big(\frac{u}{2}\Big)\Big)+c,\qquad c\in\r.$$
\end{example}
\begin{example}[Pseudosphere]
Choosing $$\xi_1(u)=e^{-u},\qquad \xi_2(u)=\arctanh(\sqrt{1-e^{-2u}})-\sqrt{1-e^{-2u}},\qquad a=0,$$ we have a pseudosphere. 
A loxodrome on this surface is parametrized by
$$\theta(u)=\cot\vartheta_0\,e^u+c,\qquad c\in\r.$$
\end{example}
\begin{example}[Twisted sphere]
If we consider $a=1$ and $\xi_1(u)=\sin u$,  $u\in (0,\pi)$, the integration of \eqref{ppca} leads  to an elliptic integral of the second kind\footnote{The
elliptic integral of the second kind is defined by
$$ E(\phi,m)=\int_{0}^{\phi}\sqrt{1-m\sin^2\theta}\,d\theta.$$} and $\xi_2(u)=E(u,-1)$. The corresponding surface is the twisted sphere (see the Figure~\ref{twisted}).
\begin{figure}[h!]
{\includegraphics[width=0.68\textwidth]{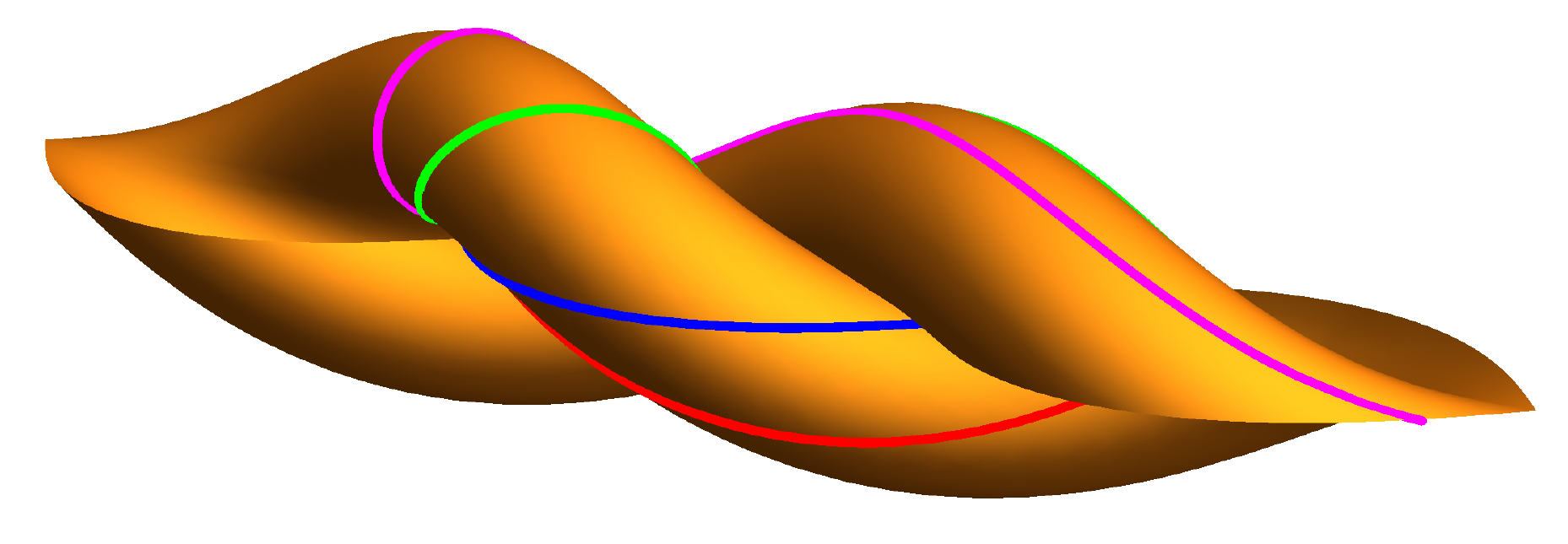}}
\begin{pspicture}(-5,0)(5,0.1)
\psset{xunit=1cm,yunit=1cm,linewidth=.03}
\put(-1.5,2.4){$\vartheta_0=\pi/4$}
\put(-2.1,3.7){$\vartheta_0=\pi/6$}
\put(3.9,0.8){$\vartheta_0=\pi/8$}
\put(-1.2,0.6){$\vartheta_0=\pi/3$}
\end{pspicture}
\caption{Four loxodromes of the twisted sphere in $\r^3$, obtained for $a =1$.}\label{twisted}
\end{figure}
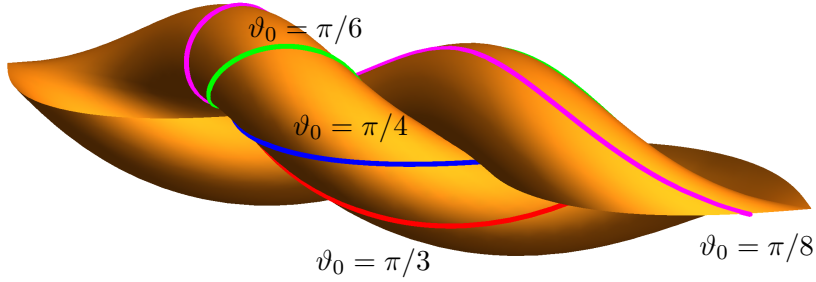
\end{example}

\begin{example}[Twisted pseudosphere]
If we consider $a=1$ and $\xi_1(u)=e^{-u}$, the integration of \eqref{ppca} leads  to 
the Gaussian hypergeometric function\footnote{The Gaussian hypergeometric function is defined for $|z|<1$ by the power series
$$\text{Hypergeometric}_2F_1[a, b, c, z]=\sum_{n=0}^\infty\frac{(a)_n\,(b)_n}{(c)_n}\frac{z^n}{n!}.$$ Here $a,b,c\in\r$, $c\notin \mathbb{Z}_{\leq 0}$ and $(q)_n$ is a Pochhammer symbol.} and $$\xi_2(u)=e^u\,\sqrt{1 - e^{-4 u}} \,(1 - 2\, \text{Hypergeometric}_2F_1[1/4, 1, 3/4, e^{4 u}]).$$ The corresponding surface is the twisted pseudosphere (see the Figure~\ref{twisted2}).
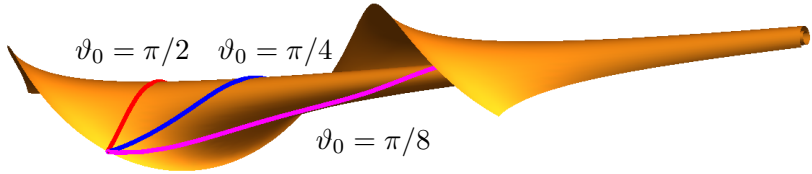
\begin{figure}[h!]
{\includegraphics[width=0.75\textwidth]{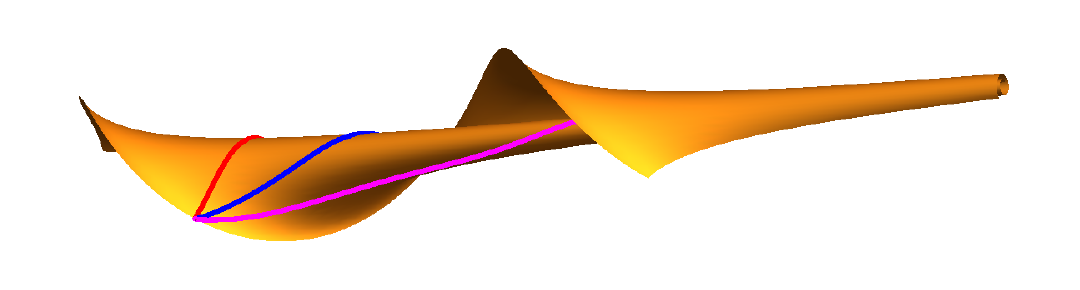}}
\begin{pspicture}(-5,0)(5,0.1)
\psset{xunit=1cm,yunit=1cm,linewidth=.03}
\put(-4.4,2.5){$\vartheta_0=\pi/2$}
\put(-2.5,2.5){$\vartheta_0=\pi/4$}
\put(-1.2,1.3){$\vartheta_0=\pi/8$}
\end{pspicture}
\caption{Three loxodromes of the twisted pseudosphere in $\r^3$, obtained for $a = 1$.}\label{twisted2}
\end{figure}
\end{example}

\section{Some remarks about loxodromes and geodesics on invariant surfaces}
In this section, we extend a classical result about loxodromes and geodesics on rotational surfaces in the Euclidean threee-space, to rotational surfaces in Bianchi-Cartan-Vranceanu spaces. We start proving the following:
\begin{theorem}\label{geo-lox}
Assume that $\alpha (s)=\psi(u(s),v(s))$ is a loxodrome, parametrized by arc length, on a $G_X$-invariant surface ${M}\subset ({N}^3,g)$  locally defined by the
$\psi(u,v)$ in  \eqref{eq-psi}, with $\vartheta_0\neq 0,\pi/2$. If $\alpha$ is a geodesic of $M$, then the surface is flat.
\end{theorem}
\begin{proof}
First we observe that from \eqref{eq-K} the Gauss curvature $K$ of an invariant surface depends only on the profile curve,
that is $K$ is constant along any orbit. This implies
that if the Gauss curvature is constant along a curve,
then either the curve is an orbit or the curve lies in
a part of the surface where the  Gauss curvature is constant.
Now, from the Clairaut's Theorem proved in \cite{MO2}, we get that the loxodrome $\alpha(s)$ is a geodesic  if, and only if,  $\omega(u(s))\cos\vartheta_0=c$, $c\in\r$. Therefore, as $\vartheta_0\neq \pi/2$, it follows that $\omega(u(s))$ is constant. Also, taking into account  the fact that $\alpha$ is not an orbit, we conclude that $\omega_{uu}(u(s))=0$ and  equation~\eqref{eq-K} implies that the Gauss curvature of $M$ is zero along the curve $\alpha$. Then, from the initial remark, we conclude that $\alpha$ lies in
a part of the invariant surface with zero Gauss curvature. 
\end{proof}
\subsection{Rotational surfaces in the Bianchi-Cartan-Vranceanu spaces}
The Bianchi-Cartan-Vranceanu spaces (BCV-spaces) are the three-dimensional Riemannian manifolds endowed with the Riemannian metrics of  the following two-parameter
family  (see \cite{Bi,Ca,Vr})
\begin{equation}\label{1.1}
 g_{\ell,m} =\frac{dx^{2} + dy^{2}}{[1 + m(x^{2} + y^{2})]^{2}} +  \left(dz +
\frac{\ell}{2} \frac{ydx - xdy}{[1 + m(x^{2} + y^{2})]}\right)^{2},\quad \ell,m \in {\r},
\end{equation}
defined on 
$N=\r^3$ if $m\geq 0$ and on $N=\{(x,y,z)\in\r^3: x^2+y^2<-1/m\}$ otherwise.
Their geometric interest lies in the
following fact: {\it the family of metrics~\eqref{1.1}
includes all three-dimensional
homogeneous metrics whose group of isometries has dimension $4$ or
$6$, except for those of constant negative sectional curvature}.
The group of isometries of these spaces contains a one-parameter subgroup isomorphic to $SO(2)$, whose infinitesimal generator is the Killing vector field given by
$$
X=-y\frac{\partial}{\partial x}+x\frac{\partial}{\partial y}.
$$
The orbits of $X$ are geodesic circles on horizontal planes with centre on the $z$-axis and the volume function
\begin{equation}\label{eq:omega}
\omega(u)=\sqrt{g_{m,\ell}(X,X)}=\frac{f(u)}{2(1+m\,f(u)^2)}\sqrt{4 +\ell^2\, f(u)^2},
\end{equation}
where $f>0$ represents the Euclidean radius of the principal orbit. 
Concerning $SO(2)$-invariant surfaces in the BCV-spaces, as a consequence of Theorem~\ref{geo-lox} we have the following result, well known in the space $\r^3$.

\begin{corollary}
Upon a rotational surface $M$ in a BCV-space a loxodrome which is neither a meridian nor a parallel cannot be a geodesic unless $M$ is a vertical cylinder.
\end{corollary}
\begin{proof}
We suppose that $\alpha (s)=\psi(u(s),v(s))$ is a loxodrome, parametrized by arc length, that is also a geodesic on the surface $M$ given by
$$\psi(u,v)=(f(u)\,\cos v,f(u)\,\sin v,g(u)).$$ 
From  Theorem~\ref{geo-lox}, we have that the surface $M$ is flat. Therefore, by using equation~\eqref{eq-main} one as  $\omega(u)=c_1\,u+c_2$, $c_1,c_2\in\r.$ Also, as $\omega(u(s)) = constant$ (see the proof of Theorem~\ref{geo-lox}) and the curve $\alpha$ is not a parallel (i.e. $u'(s)\neq 0$), we have that $\omega_u(u(s))=0$. Then, $c_1=0$ and $\omega(u)=c_2.$ Now, from~\eqref{eq:omega} we obtain (see \cite{PiuProfir}) that $\omega'(u)=0$ if, and only if,
\begin{equation}
f'(u)\,[2+(\ell^{2}-2m)\,f(u)^{2}]=0.
\end{equation}
Consequently, $f(u)=constant$ and $M$ is a vertical cylinder.
\end{proof}

%

\section{Loxodromes on invariant surfaces of $\h^2\times\r$}
Let $\h^{2}=\{(x,y)\in\r^2\,:\,y>0\}$ be the half plane model of the
hyperbolic plane and consider $\h^2\times\r$ endowed with the
product metric
\begin{equation}\label{eq:metric-h2xr}
g=\frac{dx^2+dy^2}{y^2}+dz^2.
\end{equation}

The Lie algebra of the infinitesimal isometries of the product
$(\h^2\times\r,g)$ admits the following basis of Killing vector
fields
\begin{align*}
X_1&=\frac{(x^2-y^2+1)}{2}\frac{\partial}{\partial
x}+xy\frac{\partial}{\partial
y},\\
X_2&=\frac{\partial}{\partial x},\\
X_3&=x\frac{\partial}{\partial x}+y\frac{\partial}{\partial y},\\
X_4&=\frac{\partial}{\partial z}.
\end{align*}

The class of invariant surfaces in $\h^2\times\r$ can be
divided into three subclasses according to the following

\begin{proposition}[\cite{Onnis}]\label{lem-red}
Any surface in $(\h^2\times\r,g)$ which is invariant under the action of
a one-parameter subgroup of isometries $G_X$ generated by a Killing
vector field $X=\sum_{i=1}^4a_i X_i$, $a_i\in\r$, is congruent to a
surface invariant under the action of one of the following groups:
$$
G_{14}=G_{X_1+bX_4}, \qquad G_{24}=G_{aX_2+bX_4},\qquad G_{34}=G_{X_3+bX_4},
$$
where $a,b\in\r.$
\end{proposition}

To understand the shape of an invariant surface in $\h^2\times\r$ we need to describe the
orbits of the three groups  $G_{24}, G_{34}$ and $G_{14}$.
In \cite{Onnis} it has been shown that the orbit of a point $p_0=(x_0,y_0,z_0)\in \h^2\times\r$ is:

$\bullet$ {\em under the action of $G_{24}$} the curve parametrized by
\begin{equation}\label{eq:orbitsg24}
(a\,v+x_0,y_0,b\, v +z_0),\quad v\in(-\epsilon,\epsilon),
\end{equation}
which looks like an Euclidean straight line on the plane $y=y_0$;

$\bullet$ {\em under the action of $G_{34}$} the curve parametrized by
\begin{equation}\label{eq:orbitsg34}
(e^v x_0,e^v y_0,b\, v+z_0),\quad v\in(-\epsilon,\epsilon),
\end{equation}
which belongs to a vertical plane through the $z$-axes and looks like a logarithmic curve;

$\bullet$ {\em under the action of $G_{14}$} the curve parametrized by
\begin{equation}\label{eq:orbitsg14}
(x(v),y(v),b\, v+z_0),\quad v\in(-\epsilon,\epsilon),
\end{equation}
where
$$\left\{\begin{aligned}
x(v)&=\frac{(1-x^2-y^2)\sin v+2x\,\cos
v}{(1-x^2-y^2) \cos v - 2x\sin v +1+x^2+y^2},\\
y(v)&=\frac{2y}{(1-x^2-y^2)\cos v - 2x\,\sin v+1+x^2+y^2}.
\end{aligned}\right.$$
Thus,
$$
(x(v))^2+(y(v))^2-\beta\, y(v)+1=0,\quad \beta=\frac{1+x_0^2+y_0^2}{y_0},
$$
which looks like an Euclidean helix in a right circular cylinder with Euclidean axis in the plane $x=0$.

In the following, we describe explicitly how to parametrize an invariant surface in $(\h^2\times\r,g)$ and we use \eqref{eq-integral} to parametrize the loxodromes on it.
\begin{theorem}\label{teo-h2xr}
Let ${M}$ be a $G_X$-invariant surface of $(\h^2\times\r,g)$, and
let $\tilde{\gamma}(u)=(\xi_1(u),\xi_2(u))$ be its profile curve
in the regular part of the orbit space $({\mathcal B}=\h^2\times\r/G_X,\tilde{g})$,
which is parametrized by the invariant
functions $\xi_1$ and $\xi_2$. With respect to the local para\-me\-trization
$\psi(u,v)$ given by \eqref{eq-psi} we have:
\begin{itemize}
\item[i)] If $G=G_{24}$ is the group generated by $X_2+b\, X_4$, $b\in\r$, then
 the orbit space is ${\mathcal B}=\{(\xi_1,\xi_2)\in\r^2\,:\,\xi_1>0\}$
and a loxodrome can be parametrized by $\psi(u,v(u))$, where
\begin{equation}\label{caso24}
v(u)=\int_{u_0}^u \frac{\left(b \,\sqrt{\xi_2(t)^2-\xi_2'(t)^2}\pm\cot\vartheta_0\right)\;\xi_2(t)}{\sqrt{1+b^2\, \xi_2(t)^2}}\,dt
\end{equation}
and
\begin{equation}\label{caso24-b}
\psi(u,v)=(v,\xi_2(u),b\,v-\xi_1(u)).\end{equation}
In particular, when $b=0$ (i.e. $G=G_2$) a loxodrome can be parametrized by
$$
\beta(u)=\Big(\pm \cot\vartheta_0\,\int_{u_0}^u \xi_2(t)\,dt,\xi_2(u),-\xi_1(u)\Big).
$$
\item[ii)] If $G=G_4$, the invariant surface is a right cylinder and its loxodromes are helices that can be parametrized by
$$
\beta(u)=(\xi_1(u),\xi_2(u),\cot\vartheta_0\,(u-u_0)).
$$
\item[iii)] If $G=G_{34}$ is the group generated by $X_3+b\,X_4$, $b\in\r$,
then the orbit space is ${\mathcal
B}=\{(\xi_1,\xi_2)\in\r^2\,:\,\xi_1\in(0,\pi)\}$ and a loxodrome can be parametrized by $\psi(u,v(u))$, where
\begin{equation}\label{caso34}
v(u)=\int_{u_0}^u \frac{\big(-b \,\sqrt{\sin^2\xi_1(t)-\xi_1'(t)^2}\pm\cot\vartheta_0\big)\,\sin\xi_1(t)}{\sqrt{1+b^2\,\sin^2 \xi_1(t)}}\,dt
\end{equation}
and
\begin{equation}\label{caso34-b}
\psi(u,v)=(e^v\,\cos \xi_1(u),e^v\,\sin \xi_1(u),\xi_2(u)+b\,v).
\end{equation}
\item[iv)] If $G=G_{14}$, then the orbit space
is ${\mathcal B}=\{(\xi_1,\xi_2)\in\r^2\;:\xi_1\geq 2\}$ and a loxodrome can be parametrized by
\begin{equation}\label{caso14}
v(u)=\int_{u_0}^u  
\frac{-4 b\, \sqrt{\xi_1^2(t) - 4 - \xi_1'(t)^2}\; \pm \;2 (\xi_1^2(t) - 4)\; \cot\vartheta_0}{ (\xi_1^2(t) - 4)\, \sqrt{\xi_1^2(t) - 4\,(b^2-1)}}
 \,dt
\end{equation}
and
\begin{equation}\label{caso14-b}
\psi(u,v) = \Big(\frac{\sqrt{\xi_1^2(u) -4}\,\sin v}{\sqrt{\xi_1^2(u) -4}\,\cos v -\xi_1(u)},\frac{2 }{\xi_1(u)-\sqrt{\xi_1^2(u) -4}\,\cos v},\xi_2(u) + b\, v\Big).
\end{equation}
\end{itemize}
\end{theorem}

\begin{proof}
i) We begin with the calculations for the $G_{24}$-invariant surfaces. As $X=\frac{\partial}{\partial x}+b\frac{\partial}{\partial z}\;,\; b\in\r$, a set of two invariant functions is
$$\xi_1=b\,x-z,\qquad
\xi_2=y>0.$$
Thus, the orbit space is ${\mathcal B}=\{(\xi_1,\xi_2)\in\r^2\;:\;\xi_2>0\}$ and the
orbital metric is
$$
\tilde{g}= \frac{d\xi_1^2}{1+b^2\xi_2^2}+ \frac{d\xi_2^2}{\xi_2^2}.
$$
A lift of $\tilde{\gamma}$ with respect to $\pi$ is given by
$$
\gamma(u)=(0,\xi_2(u),-\xi_1(u))
$$
and, from \eqref{eq-psi} and \eqref{eq:orbitsg24}, the corresponding $G_{24}$-invariant surface is parametrized by
$$\psi(u,v)=(v,\xi_2(u),-\xi_1(u)+b\,v).$$
Then 
\begin{equation}\label{p24}F=-b\,\xi_1'(u),\qquad G=\omega^2=\frac{1}{\xi_2^2(u)}+b^2.
\end{equation}
Also, as $\tilde{\gamma}(u)$ is parametrized by arc length, it follows that
\begin{equation}\label{p124}
\xi_1'(u)=\omega(u)\,\sqrt{\xi_2^2(u)-\xi_2'(u)^2}.
\end{equation}
Then, by substituting \eqref{p24} and \eqref{p124} in equation \eqref{eq-integral}, we obtain \eqref{caso24}.\\

iii) We consider the case of $G_{34}$-invariant surfaces in $\h^2\times\r$,  where we consider a system of  cylindrical coordinates $(r,\theta,z)$. The orbit space is given by
$${\mathcal B}=\{(\xi_1,\xi_2)\in\r^2\,:\, \xi_1\in(0,\pi)\},$$ 
where $\{\xi_1=\theta,\xi_2=z-b\ln r\}$ are $G_{34}$-invariant functions.
As we have seen, endowing the orbit space with the quotient metric
$$
\tilde{g}=\frac{d\xi_1^2}{\sin^2\xi_1}+\frac{d\xi_2^2}{1+b^2\sin^2\xi_1},
$$
the projection
$$
(r,\theta,z)\xmapsto{\pi}(\theta,z-b\ln r)
$$
becomes a Riemannian submersion. Using cylindrical coordinates, a lift of $\tilde{\gamma}$ with respect to $\pi$ is given by
$$
\gamma(u)=(1,\xi_1(u),\xi_2(u))
$$
and, from \eqref{eq-psi} and \eqref{eq:orbitsg14} it follows that the corresponding invariant surface is given by
$$\psi(u,v)=(e^v\,\cos \xi_1(u),e^v\,\sin \xi_1(u),\xi_2(u)+b\,v).$$
Thus  we obtain that
\begin{equation}\label{p} F=b\,\xi_2'(u),\qquad G=\omega^2=\frac{1+b^2\,\sin^2\xi_1(u)}{\sin^2\xi_1(u)}.
\end{equation}
 Also, as $\tilde{\gamma}(u)$ is parametrized by arc length, we get
\begin{equation}\label{p1}\xi_2'(u)=\omega(u)\,\sqrt{\sin^2\xi_1(u)-\xi_1'(u)^2}.
\end{equation}
Substituting \eqref{p} and \eqref{p1} in  equation \eqref{eq-integral}, we obtain \eqref{caso34}.\\

iv) We consider the case of $G_{14}$-invariant surfaces in $\h^2\times\r$,  again with respect to a system of  cylindrical coordinates $(r,\theta,z)$. The orbit space is given by
$${\mathcal B}=\{(\xi_1,\xi_2)\in\r^2\,:\, \xi_2 \geq 2\},$$ where 
\[
\xi_1 =  \frac{r^2 +1}{r\sin \theta} ,\qquad
\xi_2 = z-b\arctan\Big(\frac{2 r\, \cos \theta}{r^2 -1}\Big), \qquad b \in \r,
\]
 are $G_{14}$-invariant functions.
As we have seen, endowing the orbit space with the quotient metric
$$
\tilde{g}=\frac{d\xi_1^2}{\xi_1^2 - 4}+\frac{\xi_1^2 -4}{\xi_1^2 + 4 (b^2 -1)} \, d\xi_2^2,
$$
the projection
$$
(r,\theta,z)\xmapsto{\pi}\Big(\frac{r^2+1}{r\, \sin \theta},z-b\,\arctan\Big(\frac{2 r\, \cos \theta}{r^2 -1}\Big)\Big)
$$
becomes a Riemannian submersion. By using cylindrical coordinates, a lift of $\tilde{\gamma}$ with respect to $\pi$ is given by:
\[
\gamma(u) = \Big(\frac{\xi_1(u) + \sqrt{\xi_1^2(u) -4}}{2},\frac{
\pi}{2}, \xi_2(u)\Big).
\]
The corresponding invariant surface is parametrized by
$$\psi(u,v) = \Big(\frac{\sqrt{\xi_1^2(u) -4}\,\sin v}{\sqrt{\xi_1^2(u) -4}\,\cos v -\xi_1(u)},\frac{2 }{\xi_1(u)-\sqrt{\xi_1^2(u) -4}\,\cos v},\xi_2(u) + b\, v\Big).
$$ 
Then we find that
\begin{equation}\label{p14}
G = \omega^2 = \frac{\xi^2(u) -4}{4} + b^2,  \qquad F = b + \xi_2'(u).
\end{equation}
Also, as $\tilde{\gamma}(u)$ is parametrized by arc length, we have
\begin{equation}\label{p114}
\xi_2'(u)= \frac{2\,\omega(u)}{\sqrt{\xi_1^2(u) -4}} \,  \sqrt{1 - \frac{\xi_1'(u)^2}{\xi_1^2(u)- 4}}.
\end{equation}
Now we use  \eqref{p14} and \eqref{p114} in~\eqref{eq-integral} to get~\eqref{caso14}.
\end{proof}

\begin{example}
We consider the minimal surface in $\h^2\times\r$ given by the graph of the function $z=-\ln y$ (see~\cite{Onnis}). This surface is $G_{2}$-invariant and its profile curve is given by
$$
\tilde{\gamma}(u)=(u/\sqrt{2}, e^{u/\sqrt{2}}).
$$
By using \eqref{caso24-b}, the surface can be parametrized by
$$
\psi(u,v)=(v,e^{u/\sqrt{2}},-u/\sqrt{2}).
$$
Therefore from \eqref{caso24} it comes out that the loxodromes which are not orbits are given by:
$$
\beta(u)=\psi(u,\sqrt{2}\cot\vartheta_0\,e^{u/\sqrt{2}}).
$$
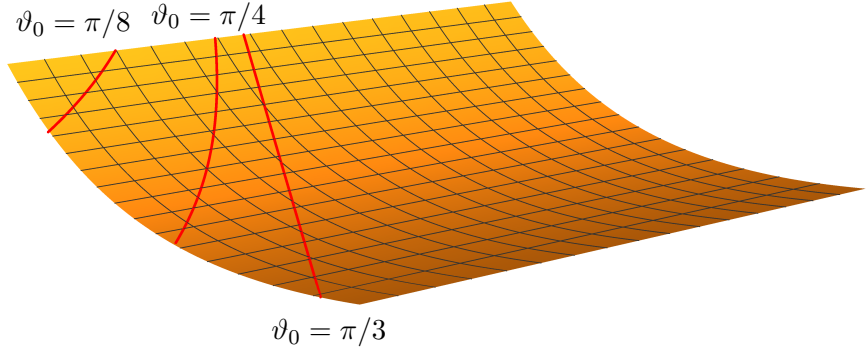
\begin{figure}[h!]
{\includegraphics[width=0.7\textwidth]{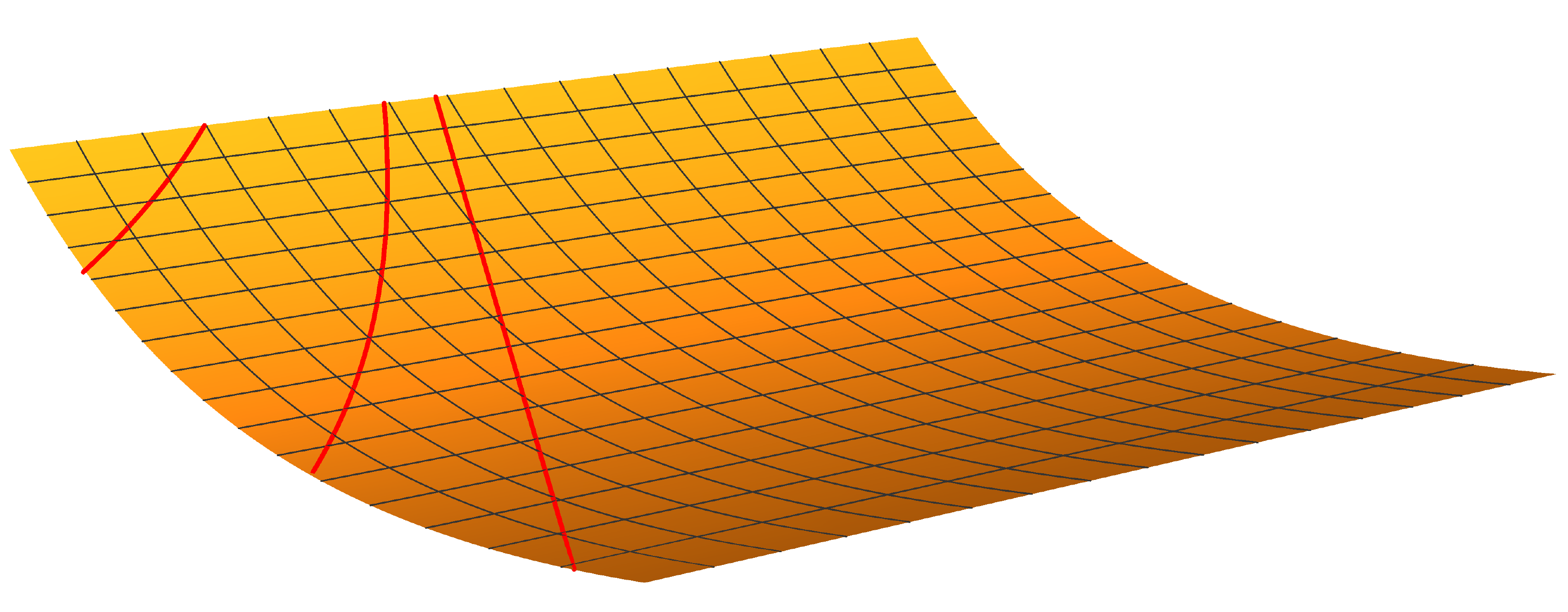}}
\begin{pspicture}(-5,0)(5,0.1)
\psset{xunit=1cm,yunit=1cm,linewidth=.03}
\put(-5.6,4.3){$\vartheta_0=\pi/8$}
\put(-3.8,4.4){$\vartheta_0=\pi/4$}
\put(-2.2,0.2){$\vartheta_0=\pi/3$}
\end{pspicture}
\caption{Three loxodromes of the $G_2$-invariant minimal graph in $\h^2\times\r$ given by $z=-\ln y$.}
\end{figure}
\end{example}

\begin{example}[The funnel surface]\label{ex:funnel}
The {\em funnel surface} is a complete minimal surface in $\h^2\times\r$ and it is the graph of the function $z=\ln(\sqrt{x^2+y^2})$ (see \cite{Onnis}). This surface is $G_{34}$-invariant and it can be obtained  starting from the simplest curve in $\h^2\times\r/G_{34}$, choosing $b = 1$ and $\xi_2=0$, whose  a parametrization by arc length is
$
\tilde{\gamma}(u)=(2 \arccot e^{-u},0).
$
Using cylindrical coordinates, a lift of $\tilde{\gamma}$ with respect to $\pi$ is given by:
$$
\gamma(u)=(1,2 \arccot e^{-u},0)
$$
and, by means of  \eqref{caso34-b}, the funnel surface can be parametrized by
$$
\psi(u,v)=(-e^v \tanh u, e^v \sech u,v).
$$
Taking \eqref{caso34} into account, we find that the loxodromes which are not orbits can be parametrized by
$$
\beta(u)=\psi\left(u,\int_{u_0}^u \frac{\pm \,\cot\vartheta_0}{\sqrt{1+\cosh^2 t}}\,dt\right).
$$
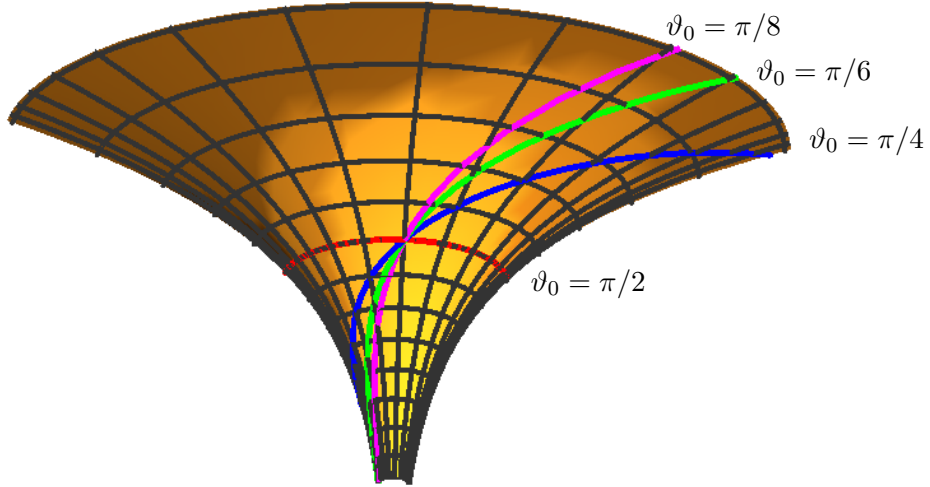
\begin{figure}[h!]
{\includegraphics[width=0.66\textwidth]{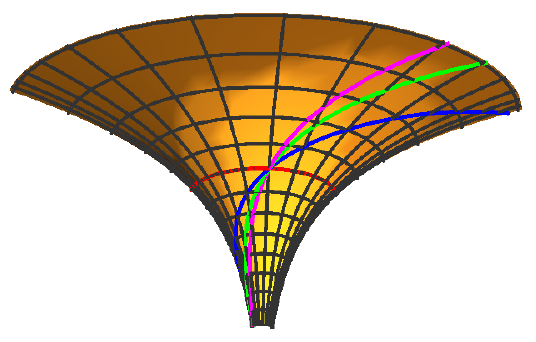}}
\begin{pspicture}(-5,0)(5,0.1)
\psset{xunit=1cm,yunit=1cm,linewidth=.03}
\put(3.5,6.7){$\vartheta_0=\pi/8$}
\put(4.7,6.1){$\vartheta_0=\pi/6$}
\put(5.4,5.2){$\vartheta_0=\pi/4$}
\put(1.7,3.3){$\vartheta_0=\pi/2$}
\end{pspicture}
\caption{Four loxodromes of the funnel surface through the point $(0,1,0)$ as seen from the point $(1,10,-4)$.}
\label{fig-funnel}
\end{figure}
\end{example}

\section{Loxodromes on invariant surfaces of the Heisenberg group $\h_3$}

We consider on the three-dimensional Heisenberg space $\h_3$, represented in $Gl_3(\r)$ by
$$
\left[
\begin{array}{ccc}
1&x&z+\frac{1}{2}x y\\ 0&1&y\\ 0&0&1
\end{array}
\right],\qquad x,y,z\in\r,
$$
 the left-invariant metric
$$g=dx^2+dy^2+\Big(\frac{1}{2}y\,dx-\frac{1}{2}x\,dy+dz\Big)^2.$$
The
isometry group of $(\h_3,g)$ has dimension $4$, which is the maximal
one for a non constant curvature three-manifold. In this case we
have the following
\begin{proposition}
The Lie algebra of the infinitesimal isometries of 
$(\h_3,g)$ admits the following basis of Killing vector fields
\begin{align*}
X_1&=\frac{\partial}{\partial
x}+\frac{y}{2}\frac{\partial}{\partial
z},\\
X_2&=\frac{\partial}{\partial
y}-\frac{x}{2}\frac{\partial}{\partial
z},\\
X_3&=\frac{\partial}{\partial z},\\
X_4&=-y\frac{\partial}{\partial x}+x\frac{\partial}{\partial y}.
\end{align*}
\end{proposition}

Also,  the one-dimensional subgroups of the isometry group
$Isom(\h_3,g)$ belong to one of the two following families (see \cite{FIMEPE}):
\begin{itemize}
  \item [1)] the one-parameter subgroups generated by linear
  combinations
  $$
  a_1X_1+a_2X_2+a_3X_3+bX_4,
  $$
 with $b\neq 0$, that are called subgroups of {\it
  helicoidal type}. If $a_i=0$, for $i\in \{1,2,3\}$,
   we obtain the group $SO(2)$ generated by $X_4$;
  \item [2)] the one-parameter subgroups generated by linear
  combinations of $X_1$, $X_2$ and $X_3$, that are called of {\it translational type}.
\end{itemize}

Therefore, a surface in the space $(\h_3,g)$ is called {\it
helicoidal} (respectively, {\it translational}) if it's invariant under the
action of a helicoidal (respectively, a translational) one-parameter subgroup
of isometries. In  \cite{FIMEPE} one finds the following 

\begin{proposition}\label{lem-red1}
A surface in $(\h_3,g)$ which is invariant under the action of a
one-parameter subgroup of isometries $G_X$ generated by a Killing
vector field $X=\sum_{i}a_i \,X_i$, $a_i\in\r$,  is congruent to a
surface invariant under the action of one of the following subgroups:
$$
G_1,\qquad G_3,\qquad G_{43}.
$$
\end{proposition}

\begin{theorem}\label{teo-h3}
Let ${M}$ be a $G_X$-invariant surface of the Heisenberg group $(\h_3,g)$, and
let $\tilde{\gamma}(u)=(\xi_1(u),\xi_2(u))$ be its profile curve
in the regular part of the orbit space $({\mathcal B}=\h_3/G_X,\tilde{g})$,
which is parametrized by the invariant
functions $\xi_1$ and $\xi_2$. With respect to the local para\-me\-trization
$\psi(u,v)$ given by~\eqref{eq-psi}, we have:
\begin{itemize}
\item[i)] if $G=G_{1}$, then the orbit space is ${\mathcal B}=\r^2$ 
and a loxodrome can be parametrized by $\psi(u,v(u))$, where
\begin{equation}\label{ca1}
v(u)=\int_{u_0}^u \frac{\xi_1(t)\sqrt{1-\xi_1'(t)^2}\pm\cot\vartheta_0}{\sqrt{1+\xi_1(t)^2}}\,dt
\end{equation}
and
\begin{equation}\label{ca1-b}
\psi(u,v)=\Big(v,\xi_1(u),\frac{\xi_1(u)}{2}v-\xi_2(u)\Big).
\end{equation}

\item[ii)] if $G=G_{3}$, then the orbit space is ${\mathcal B}=\r^2$
and a loxodrome can be parametrized by $\psi(u,v(u))$, where
\begin{equation}\label{ca3}
v(u)=- \frac{\xi_1(u)}{2} \,\int_{u_0}^u  \sqrt{1-\xi_1'(t)^2 \,dt} + \int_{u_0}^u \xi_1(t)\, \sqrt{1-\xi_1'(t)^2} \,dt  \pm\cot\vartheta_0 \,(u- u_0)
\end{equation}
and
\begin{equation}\label{ca3-b}
\psi(u,v)=(\xi_1(u),\xi_2(u),v).
\end{equation} 
These curves are general helices with axis $X_3$.  

\item[iii)] If $G=G_{43}$, then the orbit space is ${\mathcal B}=\{(\xi_1,\xi_2)\in\r^2\;:\;\xi_1\geq 0\}$ 
and a loxodrome can be parametrized by $\psi(u,v(u))$, where
\begin{equation}\label{ca-heli}
v(u)=\int_{u_0}^u \frac{(\xi_1(t)^2-2 a)\,\sqrt{1-\xi_1'(t)^2}\pm2\cot\vartheta_0\,\xi_1(t)}{\xi_1(t)\sqrt{4 \xi_1(t)^2+(\xi_1(t)^2-2 a)^2}}\,dt
\end{equation}
and
\begin{equation}\label{heli}
\psi(u,v)=(\xi_1(u)\,\cos v,\xi_1(u)\,\sin v,\xi_2(u)+a\,v).
\end{equation}
\end{itemize}
\end{theorem}
\begin{proof}
i) We start considering the case of $G_{1}$-invariant surfaces. As $$\xi_1=y,\qquad \xi_2=\frac{xy}{2}-z$$ are $X_1$-invariant functions, then the orbit space is given by
${\mathcal B}=\r^2,$  equipped with the metric
$$
\tilde{g}=d\xi_1^2+\frac{d\xi_2^2}{1+\xi_1^2}.
$$
Taking the lift of $\tilde{\gamma}$ given by
$$
\gamma(u)=(0,\xi_1(u),-\xi_2(u)),
$$
the corresponding invariant surface is parametrized as in \eqref{ca1-b}.
Thus, we have that
\begin{equation}\label{ph} F=-\xi_1(u)\,\xi_2'(u),\qquad G=1+\xi_1(u)^2.
\end{equation}
 Also, as $\tilde{\gamma}(u)$ is parametrized by arc length, it turns out
\begin{equation}\label{p1h}\xi_2'(u)=\omega(u)\,\sqrt{1-\xi_1'(u)^2}.
\end{equation}
Finally, we substitute \eqref{ph} and \eqref{p1h} in \eqref{eq-integral} and we obtain \eqref{ca1}.\\

ii) If $G=G_3$, the orbit space is $(\r^2,\tilde{g})$, where 
$
\tilde{g}=du^2+dv^2.
$
Also, a $G_3$-invariant vertical cylinder can be parametrized by \eqref{ca3-b} and $$F=\frac{\xi_1'(u)\,\xi_2(u)-\xi_2'(u)\,\xi_1(u)}{2},\qquad G =1.$$
Moreover, as $\tilde{\gamma}(u)$ is parametrized by arc length, one gets
\begin{equation}\label{p2h}\xi_2'(u)=\sqrt{1-\xi_1'(u)^2}.
\end{equation}
Then we use  \eqref{p2h} to find
\begin{equation}\label{pht}
\begin{aligned}
\int_{u_0}^u  \frac{F}{\omega^2} \,dt &= \int_{u_0}^u F \,dt \\
&=\frac{1}{2} \xi_1(u)\,\xi_2(u)  -  \int_{u_0}^u \xi_1(ut)\,\xi_2'(ut) \,dt  \\
&=\frac{1}{2}\xi_1(u)\int_{u_0}^u \sqrt{1-\xi_1'(t)^2}\, dt   -  \int_{u_0}^u \xi_1(t) \sqrt{1-\xi_1'(t)^2}\, dt.
\end{aligned}
\end{equation}
Finally, substituting \eqref{pht}  in \eqref{eq-integral}, we obtain \eqref{ca3}.\\

iii) If $X=X_4+a\,X_3$ and we use cylindrical coordinates $(r,\theta,z)$, we have that $\{\xi_1=r,\xi_2=z-a\,\theta\}$ is a set of independent invariant functions. Therefore, on the regular part of the orbit space
$$\mathcal{B}_r=\{(\xi_1,\xi_2)\in\r^2\;:\;\xi_1>0\},$$ 
we take the orbital metric
$$
\tilde{g}=d\xi_1^2+\frac{4\, \xi_1^2\,d\xi_2^2}{4 \xi_1^2+(\xi_1^2-2 a)^2}. 
$$
Then, as $\tilde{\gamma}$ is parametrized by arc length (with respect to the metric $\tilde{g}$), we get
\begin{equation}\label{p1heli}\xi_2'(u)=\omega(u)\,\frac{\sqrt{1-\xi_1'(u)^2}}{\xi_1(u)}.
\end{equation}
Choosing $\tilde{\gamma}(u)=(\xi_1(u),0,\xi_2(u))$, the helicoidal surface can be parametrided by \eqref{heli} and thus 
\begin{equation}\label{pheli} F=\frac{\xi_2'(u)\,(2a-\xi_1(u)^2)}{2},\qquad G=\frac{4 \xi_1(u)^2+(\xi_1(u)^2-2 a)^2}{4}.
\end{equation}
Making use of  \eqref{p1heli} and \eqref{pheli} in \eqref{eq-integral}, we obtain \eqref{ca-heli}.
\end{proof}

\begin{example}[The helicoidal catenoid]
The {\em helicoidal catenoid} is a helicoidal minimal surface in the Heisenberg group (see \cite{FIMEPE}), that is obtained choosing  $a=1/2$ and the profile curve given by 
$$
\tilde{\gamma}(u)=(\sqrt{u^2+1},(u-\arccot u)/2).
$$
By using \eqref{heli}, this surface can be parametrized by
$$
\psi(u,v)=\Big(\sqrt{u^2+1}\,\cos v,\sqrt{u^2+1}\,\sin v,\frac{u+v-\arccot{u}}{2}\Big)
$$
and from \eqref{ca-heli} it follows that the loxodromes which are not orbits can be parametrized by
$$
v(u)=\sqrt{2}\, (1 \pm \cot{\vartheta_0})\,\arctan\Big(\frac{u}{\sqrt{2}}\Big)-\arctan{u}. 
$$
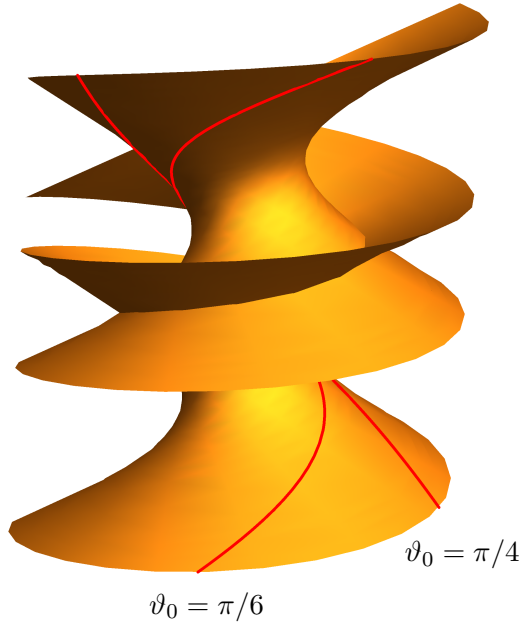
\begin{figure}[h!]
{\includegraphics[width=0.4\textwidth]{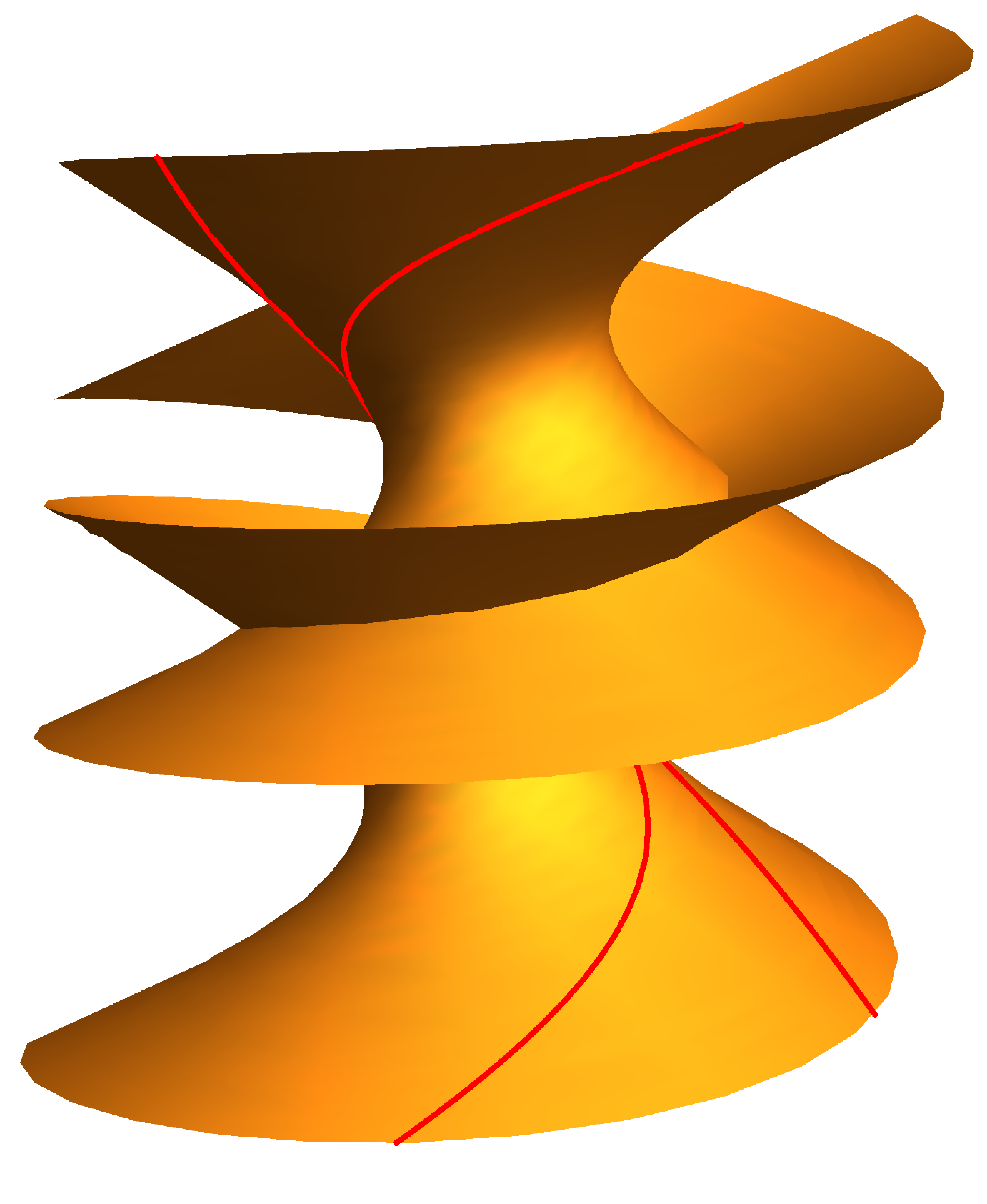}}
\begin{pspicture}(-5,0)(5,0.1)
\psset{xunit=1cm,yunit=1cm,linewidth=.03}
\put(-1.3,0.3){$\vartheta_0=\pi/6$}
\put(2.1,1){$\vartheta_0=\pi/4$}
\end{pspicture}
\caption{Two loxodromes of the helicoidal catenoid in $\h_3$.}
\end{figure}
\end{example}

\section{Loxodromes on invariant surfaces with constant Gauss curvature}

In this section we consider the case of a $G_X$-invariant surface ${M}\subset ({N}^3,g)$ such that the induced metric is of constant Gauss curvature. For this case we shall restrict our investigation to the case when
the lift $\gamma$, used to construct the parametrization of the surface \eqref{eq-psi}, is horizontal.  With this
assumption, the equation~\eqref{eq-integral} can be integrated.

\begin{proposition}[Positive curvature] Let ${M}\subset ({N}^3,g)$ be a $G_X$-invariant surface of
constant positive Gauss curvature $K=1/R^2$, locally parametrized by
$\psi(u,v)$ given by \eqref{eq-psi} with $\gamma$ horizontal lift. Then a loxodrome on ${M}$ which is not an orbit can be parametrized by
\begin{equation}
    v(u)= \mp\frac{R\,\cot\vartheta_0}{\sqrt{a}}\,\arcsinh\Big(R\,\frac{{\omega}_u(u)}{\omega(u)}\Big)+b,\qquad a,b\in\r,\, a> 0.
\end{equation}
\end{proposition}
\begin{proof}
Firstly, as $K=1/R^2$, from \eqref{eq-main} we obtain 
\begin{equation}\label{erre}
     R^2\omega_{uu} (u)+\omega (u)=0.
\end{equation}
Also, from {\eqref{erre}} follows that
$$
\frac{d}{du}\left((\omega(u))^2+R^2\,(\omega_u(u))^2\right)=2\,\omega_u(u)\,(\omega(u)+R^2\,\omega_{uu}(u))=0,
$$
which implies that there exists a constant $a\in\r,\, a> 0,$ such that
\begin{equation}\label{bi}
(\omega(u))^2+R^2\,(\omega_u(u))^2=a.
\end{equation}
Combining \eqref{erre} and \eqref{bi}, we find
\begin{equation}\label{bi1}
   {\omega}_{uu}\,\omega-{\omega}_{u}^2=-\frac{a}{R^2}.
\end{equation}
Considering the new variable
 $$
 \eta(u)=R\,\frac{{\omega}_{u}(u)}{\omega(u)},
 $$
and  taking into account \eqref{bi1}, we get
\begin{equation}\label{dw}
    d\eta=-\frac{a}{R\,\omega^2}\,du ,
\end{equation}
and also, by means of \eqref{bi},
\begin{equation}\label{dww}
    \sqrt{1+\eta^2}=\frac{\sqrt{a}}{\omega}.
\end{equation}
Finally, integrating  \eqref{eq-integral} we have
\begin{equation}\label{eqvu}\begin{aligned}
    v(u)&=\pm\cot\vartheta_0\int \dfrac{1}{\omega}\,du=
   \mp \frac{R\,\cot\vartheta_0}{\sqrt{a}}\int\frac{d\eta}{\sqrt{1+\eta^2}}\\&=\mp\frac{R\,\cot\vartheta_0}{\sqrt{a}}\,\arcsinh \eta+b\\
    &=\mp\frac{R\,\cot\vartheta_0}{\sqrt{a}}\,\arcsinh\Big(\frac{R\,{\omega}_{u}(u)}{\omega(u)}\Big)+b,\quad b\in\r.
\end{aligned}
\end{equation}
\end{proof}

\begin{proposition}[Negative curvature] Let ${M}\subset ({N}^3,g)$ be a $G_X$-invariant surface of
constant negative Gauss curvature $K=-1/R^2$, locally parametrized by
$\psi(u,v)$  given by \eqref{eq-psi}, with $\gamma$ horizontal lift. Then a loxodrome on ${M}$ which is not an orbit can be parametrized by
\begin{equation*}
  \begin{array}{lll}
   \displaystyle{v(u)=\mp\frac{\cot\vartheta_0}{{\omega}_u(u)}}+b, & \text{if} & a=0,\\
   &&\\
   \displaystyle{ v(u)=\mp\frac{R\,\cot\vartheta_0}{\sqrt{-a}}\,\ln\Big(\frac{R\,\omega_u(u)+\sqrt{-a}}{\omega(u)}\Big)+b}, & \text{if} & a<0,\\
    &&\\
   \displaystyle{ v(u)=\pm\frac{R\,\cot\vartheta_0}{\sqrt{a}}\,\arcsin\Big(\frac{R\,{\omega}_{u}(u)}{\omega(u)}\Big)+b},& \text{if} & a>0,
  \end{array}
\end{equation*}
where $b\in\r$.
\end{proposition}
\begin{proof}
As $K=-1/R^2$,  equation \eqref{eq-main} becomes
\begin{equation*}\label{erre1}
     \omega(u)-R^2\omega_{uu} (u)=0.
\end{equation*}
This implies 
\begin{equation}\label{bi2}
\omega(u)^2-R^2\,\omega_u(u)^2=a,\quad a\in\r.
\end{equation}
These two latter conditions imply
\begin{equation}\label{bi3}
   {\omega}_{uu}(u)\,\omega(u)-{\omega}_u(u)^2=\frac{a}{R^2}.
\end{equation}
In this case, the constant $a$ can be any real number.
Performing changes of variables, similar to the case of constant positive curvature, equation \eqref{eq-integral} can be integrated.\\
When $a\neq 0$,  making the change 
$$
 \eta(u)=R\,\frac{{\omega}_{u}(u)}{\omega(u)}
 $$
and taking into account \eqref{bi3}, we get
$$
    d\eta=\frac{a}{R\,\omega^2(u)}\,du.
$$
$\bullet$ For $a<0$, from \eqref{bi2} we get
$$ \sqrt{\eta^2-1}=\frac{\sqrt{-a}}{\omega}$$ 
and then
$$\begin{aligned}
     v(u)&=\pm\cot\vartheta_0\int \dfrac{1}{\omega}\,du=
   \mp \frac{R\,\cot\vartheta_0}{\sqrt{-a}}\int\frac{d\eta}{\sqrt{\eta^2-1}}\\&=\mp\frac{R\,\cot\vartheta_0}{\sqrt{-a}}\,\ln (\eta+\sqrt{\eta^2-1})+b\\
    &=\mp\frac{R\,\cot\vartheta_0}{\sqrt{-a}}\,\ln\Big(\frac{R\,\omega_u(u)+\sqrt{-a}}{\omega(u)}\Big)+b,\qquad b\in\r.
\end{aligned}
$$

$\bullet$ If $a>0$, from \eqref{bi3} we obtain
$$ \sqrt{1-\eta^2}=\frac{\sqrt{a}}{\omega}.$$ 
Consequently,
$$\begin{aligned}
     v(u)&=\pm\cot\vartheta_0\int \dfrac{1}{\omega}\,du=
   \pm \frac{R\,\cot\vartheta_0}{\sqrt{a}}\int\frac{d\eta}{\sqrt{1-\eta^2}}\\&=\pm\frac{R\,\cot\vartheta_0}{\sqrt{a}}\,\arcsin \eta+b\\
    &=\pm\frac{R\,\cot\vartheta_0}{\sqrt{a}}\,\arcsin\Big(\frac{R\,{\omega}_{u}(u)}{\omega(u)}\Big)+b,\qquad b\in\r.
\end{aligned}
$$
$\bullet$  If $a=0$, the result follows by observing that, as $\omega>0$, from \eqref{bi2} we have that $\omega_u\neq 0$;  moreover from~\eqref{bi3} it follows
\[
\frac{1}{\omega(u)} = \frac{d}{du}\Big(\frac{1}{\omega(u)_u}\Big)\] .

\end{proof}

When the Gauss curvature is zero, one has  $\omega_u=a\in\r$ and  the following 

 \begin{proposition}[Flat case] Let ${M}\subset ({N}^3,g)$ be a flat $G_X$-invariant surface, locally parametrized by
$\psi(u,v)$ given by \eqref{eq-psi} with $\gamma$ horizontal lift. Then a loxodrome on ${M}$, which is not an orbit, can be parametrized by

\begin{equation*}
  \begin{array}{lll}
   \displaystyle{  v(u)=\pm\frac{\cot\vartheta_0}{a}\,\ln\omega(u)+b}, & \text{if} & a\neq 0,\\
   &&\\
   \displaystyle{ v(u)=\pm\frac{\cot\vartheta_0}{c}\,u+b}, & \text{if} & a=0,  \end{array}
\end{equation*}
where $b, c \in\r$.
\end{proposition}

\end{document}